\documentclass[10pt]{amsart}

\usepackage{amsfonts}
\usepackage{amssymb}

\usepackage[pagebackref,hypertexnames=false, colorlinks,
citecolor=red,
linkcolor=red]{hyperref} 
\usepackage[backrefs]{amsrefs}

\theoremstyle{plain}
\newtheorem{thm}{Theorem}[section]
\newtheorem{prop}[thm]{Proposition}
\newtheorem{cor}[thm]{Corollary}
\newtheorem{lemma}[thm]{Lemma}
\theoremstyle{definition}
\newtheorem{defn}{Definition}[section]

\theoremstyle{remark}

\newcommand{\abs}[1]{\left\lvert #1 \right\rvert}
\newcommand{\inp}[2]{\left\langle #1,#2 \right\rangle}
\newcommand{\norm}[1]{\left\| #1 \right\|}
\newcommand{\cl}[1]{\overline{#1}}

\newcommand{\cal}[1]{\mathcal{#1}}

\DeclareMathOperator{\linspan}{span}

\newcommand{\wk}{\text{weak$^\ast$}}
\DeclareMathOperator*{\mult}{mult}
\DeclareMathOperator{\trace}{trace}

\newcommand{\calA}{\mathcal{A}}
\newcommand{\calM}{\mathcal{M}}

\newcommand{\calF}{\mathcal{F}}
\newcommand{\calI}{\mathcal{I}}

\newcommand{\bbC}{\mathbb{C}}
\newcommand{\bbD}{\mathbb{D}}
\newcommand{\bbN}{\mathbb{N}}

\newcommand{\bbR}{\mathbb{R}}

\newcommand{\bbT}{\mathbb{T}}

\newcommand{\alg}[1]{\mathcal{#1}}

\newcommand{\hs}[1]{\mathcal{#1}}
\newcommand{\bh}[1]{B(\hs #1)}

\newcommand{\tc}[1]{TC(\hs #1)}
\newcommand{\hilbs}[1]{HS(\hs #1)}
\newcommand{\m}[1][]{\mult}
\newcommand{\mh}[1]{\mult(#1)}
\newcommand{\mhk}[2]{\mult(#1,#2)}

\newcommand{\finset}[2][n]{\{#2_1,\ldots,#2_#1\}}
\newcommand{\finseq}[2][n]{(#2_1,\ldots,#2_#1)}
\allowdisplaybreaks

\begin{document}

\title[Duality, Tangential Interpolation and T\"oplitz Corona
Problems]{Duality, Tangential Interpolation, \\ and
  T\"oplitz Corona Problems.}

\author[M. Raghupathi]{Mrinal Raghupathi$^{\dagger}$} \date{\today}

\address{Department of Mathematics, Vanderbilt University \\
  Nashville, Tennessee, 37240 USA}

\email{mrinal.raghupathi@vanderbilt.edu}

\urladdr{http://www.math.vanderbilt.edu/~mrinalr} \thanks{$\dagger$ The first
  author was supported in part by a National Science Foundation Young
  Investigator Award, Workshop in Analysis and Probability, Texas
  A{\&}M University}

\author[B. D. Wick]{Brett D.~Wick$^{\ddagger}$} 

\address{School of Mathematics, Georgia Institute of Technology\\ 686
  Cherry Street\\ Atlanta, GA 30332-1060 USA}

\email{wick@math.gatech.edu} 

\urladdr{http://people.math.gatech.edu/~bwick6/} \thanks{$\ddagger$ The second
  author is supported by National Science Foundation CAREER Award
  DMS\# 0955432 and an Alexander von Humboldt Fellowship.}

\subjclass[2000]{Primary 47A57; Secondary 30H30, 30H50}

\keywords{Nevanlinna-Pick interpolation, Toeplitz Corona Problem,
  Distance Formulae, Hilbert Module}

\begin{abstract}
  In this paper we extend a method of Arveson~\cite{Arveson} and
  McCullough~\cite{Mc} to prove a tangential interpolation theorem for
  subalgebras of $H^\infty$. This tangential interpolation result
  implies a T\"oplitz corona theorem. In particular, it is shown that
  the set of matrix positivity conditions is indexed by cyclic
  subspaces, which is analogous to the results obtained for the ball
  and the polydisk algebra by Trent-Wick~\cite{TW} and
  Douglas-Sarkar~\cite{DS}.
\end{abstract}

\maketitle

\section{Introduction}
The classical corona problem asks whether the set of point
evaluations, for points in the unit disk $\bbD$, is dense in the
maximal ideal space of $H^\infty$. A famous result of
Carleson~\cite{C} shows that they are dense. Let $\cal A$ be an
abelian Banach algebra, and let $M$ be its maximal ideal space. A
subset $X$ of $M$ is dense in $M$ if and only if for any finite set of
functions $f_1,\ldots,f_n$ such that $\sum_{j=1}^n \abs{f_j(x)}^2 \geq
\delta^2>0$ for $x\in X$, there exists a set $g_1,\ldots,g_n\in \alg
A$ such that $f_1g_1+\cdots+f_ng_n = 1$.

Arveson, \cite{Arveson}, studied a related problem replacing the
condition $\sum_{j=1}^n \abs{f_j(x)}^2 \geq \delta^2$, by the operator
theoretic condition $\sum_{j=1}^n T_{f_j}T_{f_j}^* \geq \delta^2I$,
where $T_f$ is the T\"oplitz operator with symbol $f$ acting on the
Hardy space $H^2$. He showed that under this assumption there exists
$g_1,\ldots,g_n\in H^\infty$ such that $\sum_{j=1}^n f_j g_j =1$ and
$\sum_{j=1}^n \abs{g_j(z)}^2 \leq \delta^{-2}$. The constant
$\delta^{-2}$ is optimal, as demonstrated by the choice $f_1 = 1$.
See Schubert for the best possible constant \cite{Sc}.

In general determining the best constants in the corona problem is
considerably challenging. For the T\"oplitz corona problem we
do obtain the optimal constant. However, we make stronger assumptions.

The objective of this paper is to show that a similar T\"oplitz corona
theorem holds for the case where $\alg A$ is \wk-closed subalgebra of
$H^\infty$. Our result makes use of a modification of the Arveson
distance formula~\cite[Theorem 1]{Arveson}, a refinement of this due
to McCullough~\cite[Theorem 1]{Mc}.  These modifications allow us to
then demonstrate the first main result of this paper, a tangential
interpolation theorem for unital \wk-closed algebras $\alg A$.

\subsection{Notation}
We denote by $L^p$ the Lebesgue space on the unit circle with
respect to normalized arc-length measure. The corresponding Hardy
space will be denoted $H^p$.

Given a subset $\cal S$ of a Hilbert space $\hs H$, we denote by
$[\cal S]$ the smallest closed subspace that contains $\cal S$.

A function $g\in H^2$ is called outer if the closure of $H^\infty g$
is $H^2$. In this paper we adopt the following notation.  Given a
subalgebra $\alg A\subseteq H^\infty$ and an outer function $g$ we let
$K_g$ be the reproducing kernel of the subspace $[\alg A g]$, viewed
as a subspace of $H^2$.

\subsection{Statement of main results}
This paper is concerned primarily with tangential interpolation
theorems and their application to T\"oplitz corona problems. We would
like to give an overview of the main results. We begin by stating the
tangential interpolation problem and our main result, which is
Theorem~\ref{interp}. In Section~\ref{tangent} we will elaborate on
the connections between the two problems.

Let $\alg A$ be a unital, \wk-closed subalgebra of $H^\infty$. Let
$(x_1,\ldots,x_n)$ be a sequence of points in the unit disk $\bbD$,
let $(v_1,\ldots,v_n)$ be a sequence of vectors in $\ell^2$ and let
$(w_1,\ldots,w_n)$ be a sequence of scalars.  We identify a
function $F:\bbD\to \ell^2$ with a sequence of functions
$(f_k)_{k=1}^\infty$ in the usual way. Let $F:\bbD \to \ell^2$ be such
that $f_k\in \alg A$ for all $k\geq 1$. The function $F$ induces an
operator $M_F : H^2\to H^2\otimes \ell^2$ given by $M_F(h) =
Fh$. Hence, $M_F$ can be identified with the column operator
$(T_{f_1},\ldots)^t$. Similarly, there is a map from $H^2\otimes
\ell^2\to H^2$ given by $M_F(h_k) = \sum_{k=1}^\infty f_k h_k$. In
this case, the operator $M_F$ is identified with the row operator
$(T_{f_1},\ldots)$. We denote by $C(\alg A)$ the set of $F$ such that
$f_k\in \alg A$, for $k\geq 1$, viewed as column operators. When viewed
as row operators we use the notation $R(\alg A)$. In both instances
the norm of $M_F$, as an operator, coincides with $\sup_{z\in \bbD}
\norm{F(z)}_{\ell^2}$.

We are concerned with the following extremal problem
\[
\inf \left\{\sup_{z\in \bbD} \norm{F(z)}_{\ell^2}\,:\, \inp{v_j}{F(x_j)} =
\cl{w_j}\text{ for }j=1,\ldots,n\right\}.
\]
We say that a function $F$ such that $\inp{v_j}{F(x_j)} = \cl{w_j}$
for $j=1,\ldots,n$ is a solution to the tangential interpolation
problem.

Our main result is a characterization, in terms of matrix positivity
conditions, for the existence of a solution $F\in C(\cal A)$ such that
$\sup_{z\in \bbD} \norm{F(z)}_{\ell^2} \leq \alpha$, where $\alpha$ is
a prescribed constant.
\begin{thm}\label{interp}
  Let $\alg A$ be a unital \wk-closed subalgebra of $H^\infty$. Let
  $(x_1,\ldots,x_n)$ be a sequence of points from the unit disk
  $\bbD$, let $(v_1,\ldots,v_n)$ be a sequence of $\ell^2$ vectors,
  and let $(w_1,\ldots,w_n)$ be a sequence of scalars. Let $Q_g$
  denote the matrix
\begin{equation}
Q_g = \left[(\alpha^2 \inp{v_j}{v_i} - w_i\cl{w_j} )K_g(x_i,x_j)\right].
\end{equation}
Then there exists a function $F :\bbD \to \ell^2$ such that
$\sup_{z\in \bbD} \norm{F(z)}_{\ell^2}\leq \alpha$ and
$\inp{v_j}{F(x_j)} = \cl{w_j}$ if and only if $Q_g\geq 0$ for all
outer functions $g$ such that $\norm{g}_2 = 1$.
\end{thm}

A more careful examination of the proof of Theorem~\ref{interp}, which
will be given in Section~\ref{maintangent}, shows that we need
only consider a subset of the set of all outer functions.

Before we state this corollary we introduce some notation.  Given
$\calA$ we let $L^\infty(\calA)$ denote the smallest \wk-closed
subalgebra of $L^\infty$ that contains $\calA + \cl{\calA}$. The
algebra $L^\infty(\calA)$ is the algebra of essentially-bounded
measurable functions for some sub-sigma-algebra of the Lebesgue
measurable sets on the circle. Therefore, there exists a sigma-algebra
$\calM$ consisting of Lebesgue measurable subsets of $\bbT$ such that
$L^\infty(\calA) = L^\infty(\bbT, \calM,dm)$, where $m$ is Lebesgue
measure. We let $L^p(\cal A)$ denote the corresponding $L^p$ space.

If $g$ is an outer function, then we denote $\hs H_g = [\alg A
g]$. Recall that the kernel function for this space is $K_g$. When
$g=1$ we denote $\hs H_g$ by $\hs H$ and the corresponding kernel is
denoted $K$.
\begin{cor}\label{interpcor}
Retaining the notation of Theorem~\ref{interp}. 
\begin{enumerate}
\item There exists a function $F\in C(\alg A)$ such that $\sup_{z\in
    \bbD}\norm{F(z)}_{\ell^2} \leq \alpha$ and $\inp{v_j}{F(x_j)} =
  \cl{w_j}$ for $j=1,\ldots, n$ if and only if $Q_g\geq 0$ for all
  outer functions $g$ such that $\abs{g}\in L^2(\alg A)$ and
  $\norm{g}_2 = 1$.
\item For $F\in R(\cal A)$ let $M_{F,g} \in B(\hs H_g\otimes
  \ell^2,\hs H_g)$ be given by $h \to Fh$. If there is a constant
  $\delta>0$ such that $M_{F,g}M_{F,g}^*\geq \delta^2 I$ for all outer
  functions $g$ such that $\abs{g}\in L^2(\calA)$ and $\norm{g}_2=1$,
  then there exists a function $G$ in $C(\calA)$ such that $FG = 1$
  and $\sup_{z\in\bbD}\norm{G(z)} \leq \delta^{-1}$.
\end{enumerate}
\end{cor}

The difficulty with Theorem \ref{interp} is that the positivity
condition is over a whole family of outer functions or kernels. In
some applications we would rather have the condition over just
a single kernel. We turn to establishing a tangential
interpolation theorem where we replace the family of conditions
$Q_g\geq 0$ for all outer functions $g$ such that $\abs{g}\in L^2(\cal
A)$ by a single positivity condition. However, we can not guarantee a
solution of optimal norm.  This leads to the second main result of the
paper.

\begin{thm}
\label{onegenerator}
  Let $\cal A$ be a unital \wk-closed subalgebra of $H^\infty$, and
  let $\hs H_g$ be the subspace generated by $\cal A g$, where $g$ is
  an outer function.  Suppose that
  for each outer function $g$ such that $\abs{g}\in L^2(\calA)$ there
  exists a similarity $S_g : \hs H \to \hs H_g$. Also assume that
  there is a constant $c$, that is independent of $g$, such that
  $\norm{S_g}\|S_g^{-1}\|\leq c$ for all such outer functions $g$.
\begin{enumerate}
\item If $[(\alpha^2\inp{v_j}{v_i}-w_i\cl{w_j})K(x_i,x_j)]\geq 0$,
  then there exists $F \in C(\alg A)$ such that $\sup_{z\in
    \bbD}\norm{F(z)}_{\ell^2} \leq \alpha c$ and $\inp{v_j}{F(x_j)} =
  \cl{w_j}$ for $j=1,\ldots,n$.
\item If $M_FM_F^*\geq \delta^2$ on $B(\hs H)$, then there exists
  $G\in C(\cal A)$ such that $FG = 1$, and $\norm{G}_{C(\calA)}\leq
  c\delta^{-1}$.
\end{enumerate}
\end{thm}

The outline of the paper is as follows.  In Section~\ref{tangent} we
give background to the tangential interpolation problem, including
standard background for reproducing kernel spaces.  In Section
\ref{distance} we provide an extension of Arveson's Distance formula
needed in our context.  Section \ref{maintangent} puts the
computations and ideas from the first two sections together to prove
Theorem \ref{interp}.  Finally, in Section \ref{similar} we prove
Theorem \ref{onegenerator}, which essentially follows from Theorem
\ref{interp}, and then collect applications of Theorem
\ref{onegenerator} to the case of bounded analytic functions on
Riemann surfaces. This application generalizes a result of Ball~\cite{B}.

The corona problem and its variant the T\"oplitz corona problem have
been studied extensively in the past. The paper of
Agler-McCarthy~\cite[Section 7]{aglmc2} provides an excellent overview
of the connection between matrix positivity conditions, families of
kernels and corona problems. The connections between interpolation
theory and To\"eplitz corona problem for the bidisk and the
Schur-Agler class are described in Agler-McCarthy~\cite{aglmc} and
Ball-Trent~\cite{BT}.

\section{The tangential interpolation problem}\label{tangent}
In order to state our results and describe our setting we will need
the terminology of reproducing kernel Hilbert spaces. We begin with a
brief description. The reader should consult the text of Agler and
McCarthy~\cite{aglmcbook}, or the paper of Aronszajn \cite{Ar}.

\subsection{Reproducing kernel Hilbert spaces}
Let $X$ be a set and let $\hs C$ be a Hilbert space. We denote by $
\calF(X,\hs C)$ the set of functions from $X$ to $\hs C$. A subset $\hs H
\subseteq \calF(X,\hs{C})$ is called a $\hs C$-valued reproducing kernel
Hilbert space (RKHS) if $\hs H$ is a Hilbert space and for each $x\in X$, the
evaluation map $E_x: \hs{H} \to \hs C$ defined by $f\mapsto f(x)$ is a
bounded linear map on $\hs H$. The kernel function of $\hs H$
is the map $K:X\times X\to B(\hs C)$ defined by $K(x,y) = E_xE_y^*\in
B(\hs C)$. It is straightforward that the kernel function of $\hs H$
is a positive semidefinite function on $X\times X$ and that the span
of $\{E_x\xi\,:\, x\in X,\xi\in \hs C\}$ is dense in $\hs H$.

Conversely, every $B(\hs C)$-valued positive semidefinite function $K$
on $X\times X$ gives rise to a $\hs C$-valued RKHS $\hs H(K)$ in a canonical way and this correspondence is
one-to-one~\cite{aglmcbook}. 

We denote the Hilbert space associated to $K$ by $\hs H(K)$. We
suppress the kernel function, when the context is clear.

For $i=1,2$, let $K_i$ be a $\hs C_i$-valued kernel function on $X$
and let $\hs H_i = \hs H(K_i)$. Given a function $F:X \to \calF(\hs
C_1,\hs C_2)$ and a function $g:X\to \hs C_1$, let $Fg$ denote the
pointwise product of $F$ and $g$. We say that $F:X\to \calF(\hs
C_1,\hs C_2)$ is a multiplier from $\hs H_1$ to $\hs H_2$ if and only
if $Fg\in \hs H_2$ for all $g\in \hs H_1$. We denote the set of
multipliers from $\hs H_1$ to $\hs H_2$ by $\mhk{\hs H_1}{\hs
  H_2}$. Since the space $\hs H_i$ is completely determined by its
kernel function $K_i$ we also use the notation $\mhk{K_1}{K_2}$ to
denote the space of multipliers.

The closed graph theorem shows that the operator $M_F:\hs H_1\to
\hs H_2$ defined by $M_F(g) = Fg$ is bounded. The multiplier norm
of a function $F\in \mhk{K_1}{K_2}$ is defined as $\norm{F}_{\mult(K_1,K_2)} :=
\norm{M_F}_{B(\hs H_1,\hs H_2)}$. 

If $E_{i,x}$ denotes the evaluation map on $\hs H_i$, and $F\in
\mult(K_1,K_2)$, then
\[E_{2,x} M_F = F(x)E_{1,x}\text{ for all }x\in X.\] 

If $\hs H$ is a scalar-valued RKHS, then the evaluation map
$E_x$ is a linear functional and the unique element $k_x\in \hs H$
such that $f(x) = E_x(f) = \inp{f}{k_x}$ for all $f\in \hs H$ is
called the kernel function at the point $x$ for $\hs H$. In this case $K(x,y) = E_xE_y^\ast = \inp{k_y}{k_x}$. 

Given two scalar-valued RKHSs $\hs H_i$, for $i=1,2$, with kernel
functions $K_i$, and $f\in \mhk{\hs H_1}{\hs H_2}$, we have
$M_f^*k_{2,x} = \cl{f(x)}k_{1,x}$, where $k_{i,x}$ denotes the kernel
function for $\hs H_i$ at the point $x$.

Given a scalar-valued RKHS $\hs H(K)$ we give the Hilbert space tensor
product $\hs H \otimes \hs C$ the structure of a $\hs C$-valued RKHS
by defining $E_x(f\otimes \xi) = f(x)\xi$. A short calculation reveals
that the kernel function for $\hs H\otimes \hs C$ is $K(x,y)I_{\hs
  C}$, where $I_{\hs C}$ is the identity map on $\hs C$, and that
$E_x^*\xi = k_x\otimes \xi$.

If $F\in \mhk{\hs H}{\hs H\otimes \hs C}$, then, for each $x\in  X$,
$F(x) \in B(\bbC,\hs C)$. We identify $B(\bbC, \hs C)$ with $\hs C$
via the correspondence $T\mapsto T(1)$. We have,
\begin{align*}
\inp{M_F^*(k_x\otimes \xi)}{h} & = \inp{k_x\otimes \xi}{Fh} = \inp{k_x}{h}F(x)^*\xi = \inp{\inp{\xi}{F(x)}k_x}{h}. 
\end{align*}
Therefore, 
\begin{equation}\label{mfstar}
M_F^*(k_x\otimes \xi) = (F(x)^*\xi) k_x = \inp{\xi}{F(x)}k_x.
\end{equation} 

Given $F\in \mhk{\hs H\otimes \hs C}{\hs H}$ we have $F(x)\in B(\hs
C,\bbC)$ and so $F(x)^* \in B(\bbC,\hs C) = \hs C$, under our
identification. In this case we have,
\begin{align*}
  \inp{h\otimes \xi}{M_F^* k_x}& = \inp{F(h\otimes \xi)}{k_x} \\
& = h(x)F(x)\xi = \inp{h}{k_x}\inp{\xi}{F(x)^*} \\
& = \inp{h\otimes \xi}{k_x\otimes F(x)^*}. 
\end{align*}
Hence, 
\begin{equation}
\label{mfstar2}
M_F^*k_x = k_x\otimes F(x)^*.
\end{equation}

In this paper we will be interested primarily in two special cases: the
case where $\hs C_1 = \hs C_2 = \bbC$, and the case where either $\hs
C_1$ or $\hs C_2$ is $\ell^2 := \ell^2(\bbN)$ and the other is $\bbC$. 

We can view an operator $T \in B(\hs H,\hs H\otimes \ell^2)$ as a column
operator matrix of the form $T = \begin{bmatrix} T_1 \\ T_2 \\
  \vdots \end{bmatrix}$. If $\alg A\subseteq B(\hs H)$, then we denote by
$C(\alg A)$ the set of column operators $[T_i]$ such that $T_i\in \alg A$ for
all $i$. There is a similar identification of $B(\hs H\otimes \ell^2, \hs H)$
with the set of row operator matrices, and we denote the set of row
operators with entries from $\alg A$ by $R(\alg A)$. It is easily checked that
$\mult(\hs H,\hs H\otimes \ell^2) = C(\mult(\hs H))$ and that
$\mult(\hs H\otimes\ell^2,\hs H) = R(\mult(\hs H))$.

\subsection{Tangential interpolation}
We now describe the tangential interpolation problem. Let $X$ be a set
and let $K$ be a kernel on $X$. We will assume that $K(x,x)\not = 0$,
for all $x\in X$.

Given a finite sequence of points $(x_1,\ldots,x_n)\in X^n$, a
sequence of vectors $(v_1,\ldots,v_n) \in \hs C^n$, and a sequence of
scalars $(w_1,\dots,w_n)$. We say that a function $F\in \mult(\hs H,
\hs{H}\otimes \hs{C})$ is a solution to the associated tangential
interpolation problem if and only if $F(x_j)^*v_j = \cl{w_j}$ for
$j=1,\ldots,n$.

Given a constant $\alpha$ we are interested in finding necessary and
sufficient conditions for the existence of a solution of norm at most
$\alpha$, that is, a multiplier $F$ such that $\norm{F}_{\mhk{\hs
    H}{\hs H\otimes \hs C}} \leq \alpha$ and
\[
F(x_j)^*v_j = \overline{w_j}\text{ for }j=1,\ldots,n.
\]

As is the case with many complex interpolation problems of this
type, there is a necessary matrix positivity condition, which we now
derive. Let $F$ be a solution to the above problem and let $\alpha \geq
\norm{F}_{\mhk{\hs H}{\hs H\otimes \hs C}}$. For $x,y\in X$ and
$\xi,\zeta\in \hs C$ we have,
\begin{align}\label{matpos4}
  \inp{M_FM_F^*(k_y\otimes\zeta)}{k_x\otimes \xi} & = \inp{M_F^*(k_y\otimes \zeta)}{M_F^*(k_x\otimes \xi)} \\
\notag  & = \inp{(F(y)^*\zeta)k_y}{(F(x)^*\xi)k_x} \\
\notag  & = (F(y)^*\zeta)(\overline{F(x)^*\xi})K(x,y)
\end{align}
Consider the restriction of $M_F^*$ to the subspace $\hs
K$, that is the span of the vectors $\{k_{x_1}\otimes
v_1,\ldots,k_{x_n}\otimes v_n\}$. An element of $\hs K$ is of the form
$ k = \sum_{j=1}^n c_j k_{x_j}\otimes v_j$. Since $M_F^*$ has norm at
most $\alpha$ we see that $\inp{(\alpha^2I - M_FM_F^*)k}{k}\geq
0$. Substituting for $k$, using~\eqref{matpos4} and the fact that
$F(x_j)^*v_j = \cl{w_j}$, shows us that $\norm{M_F^*|_{\hs K}}\leq
\alpha $ if and only if
\begin{equation}\label{matpos3}
[(\alpha^2\inp{v_j}{v_i} - w_i\cl{w_j})K(x_i,x_j)]\geq 0.
\end{equation}

If $\hs C = \bbC$, and $v_i = 1 \in \bbC$, then the above positivity
condition is a necessary condition for the existence of a function
$f\in \mult(K)$ such that $\norm{f}_{\mult(K)}\leq \alpha$ and $f(x_j)
= w_j$ for $j=1,\ldots,n$. If $K(z,w) = (1-z\cl{w})^{-1}$ is the
Szeg\"o kernel for the unit disk $\bbD$, then the associated
multiplier algebra is $H^\infty(\bbD)$ and the multiplier norm is the
supremum norm. This is the setting of the classical Nevanlinna-Pick
theorem. In this case, it is a well-known fact that the necessary matrix
positivity condition is also sufficient.

In general, a single matrix positivity condition is not sufficient to
guarantee that there exist solutions of norm at most a given constant
$\alpha$. However, in certain situations we may be able to replace a
single kernel function by a set of kernel functions $\{K_\lambda\,:\,
\lambda\in \Lambda\}$ such that $\mh{K_\lambda} = \mh{\hs H}$. In
addition, this collection of kernel functions has the property that the
condition $Q_\lambda = [(\alpha^2-w_i\cl{w_j})K_\lambda(x_ix_j)] \geq
0$ for all $\lambda\in\Lambda$ is equivalent to the existence of a
multiplier $f\in \mult(\hs H)$ such that $\norm{f}_{\mult(\hs H)}\leq
\alpha$ and $f(x_j) = w_j$. This is, in fact, the situation when we
replace the algebra $H^\infty(\bbD)$ by a \wk-closed subalgebra $\alg
A$ of $H^\infty(\bbD)$~\cite{Ra1}.

To make this formal, we introduce the following definition.

\begin{defn}
  Let $\hs H$ be a scalar-valued RKHS on a set $X$. A set of kernels
  $\{K_\lambda\,:\, \lambda\in \Lambda\}$ on $X$ such that
  $\mult(K_\lambda)\supseteq \mult(\hs H)$ has the \textit{tangential
    interpolation property} for $\mult(\hs H)$ if and only if for every finite sequence
  of points $\finseq{x}$ from $X$, vectors $\finseq{v}$ from $\hs C$,
  and scalars $\finseq{w}$ the condition
\begin{equation}\label{matpos}
Q_\lambda :=
[(\alpha^2\inp{v_j}{v_i}-w_i\cl{w_j})K_\lambda(x_i,x_j)]_{i,j=1}^n\geq 0
\text{ for all }\lambda\in\Lambda
\end{equation}
implies the existence of a multiplier $F\in \mult(\hs H,\hs H\otimes
\hs C)$ such that
$$
\norm{F}_{\mult(\hs H, \hs{H}\otimes \hs{C})}\leq\alpha
$$ 
and $F(x_j)^*v_j = \cl{w_j}$ for $j=1,\ldots,n$.
\end{defn}

\subsection{Reformulation as a distance problem}
Suppose that we are given the data of tangential interpolation
problem, that is, a sequence of vectors $\finseq{v}$ in $\ell^2$, a
sequence of points $\finseq{x}$ in $X$ and a sequence of scalars
$\finseq{w}$. Let us assume that there are solutions to
this problem, that is, multipliers $F\in \mult(\hs H,\hs H\otimes
\hs C)$ such that $F(x_j)^*v_j = \cl{w_j}$ for $j=1,\ldots,n$. Let $\cal
J$ denote the set of functions $G$ in $\mult(\hs H, \hs H\otimes
\hs C)$ such that $G(x_j)^*v_j =0$ for $j=1,\ldots,n$. Given two
solutions $F_1,F_2$ to the tangential interpolation problem we see
that the difference $F_1-F_2\in \cal J$. Conversely, every solution
must lie in $F+\cal J$. Hence, the set of solutions to the
interpolation problem is precisely $F+\cal J$, where $F$ is one
particular solution. 

We will show in Section~\ref{maintangent} that under the assumption
$K(x,x)\not =0$ for all $x\in X$, the multiplier algebra $\mult(K)$ is
a unital, \wk-closed subalgebra of $B(\hs H(K))$. We will also
establish the fact that the evaluation map on $\mult(K)$, given by
$f\mapsto f(x)$, is \wk-continuous.

Since the algebras we are interested in are $\wk$-closed the least
norm of any solution to the interpolation problem is given by
$\inf\{\norm{F+G}_{\mult(\hs H,\hs H\otimes \hs C)}\,:\, G\in \cal
J\}$. This is the distance from $F$ to $\cal J$, that is,
$d(F,\cal J)$. The problem of determining necessary and sufficient
conditions for the existence of a solution of norm at most $\alpha$ is
reduced to the problem of computing a formula for the distance
$d(F,\cal J)$. With this in mind we present a refinement of a
distance formula in~\cite{Arveson,Mc} in the next section.

We will show how to deduce a T\"oplitz corona theorem from a
tangential interpolation theorem in Section~\ref{maintangent}.

\section{A Distance Formula}
\label{distance}
In this section we recall some basic facts about the distance of an
operator $A\in B(\hs H_1,\hs H_2)$ from a $\wk$-closed subspace. In
the next section we will use this formula to compute the distance of a
solution to the subspace $\cal J$ and thus obtain a tangential
interpolation theorem. This a simple application of standard duality
techniques. 

We begin by recalling some basic facts related to the dual Banach
space structure of $\bh{H}$.  Let $\hs H$ be a separable Hilbert
space. Given an operator $T\in \bh{H}$, the trace of $T$ is defined
by
\[
\trace(T) = \sum_{j=1}^\infty \inp{Te_j}{e_j}
\]
where $\{e_j\,:\, j\in \mathbb{N}\}$ is an orthonormal basis. The sum
in the definition of the trace does not depend on the choice of
orthonormal basis. An operator is called a trace class operator if and
only if $\trace(T)$ is finite. The set of trace class operators is an
ideal in $\bh{H}$ and is a Banach space in the trace class norm
$\norm{T}_1 = \trace(\abs{T})$. We let $\tc{H}$ denote the space of
trace class operators on $\hs H$. It is well known that the dual of
$\tc{H}$ can be identified naturally with $\bh{H}$ and that the dual
pairing is given by
\[\inp{T}{A} = \trace(TA),\text{ where } A\in \bh{H}, T\in \tc{H}.\]
This pairing also identifies $\tc{H}$ with the set of \wk-continuous
linear functionals on $\bh{H}$.  An operator $H\in \bh{H}$ is called a
Hilbert-Schmidt operator if and only if $HH^*\in \tc{H}$, that is,
$\sum_{j=1}^\infty \inp{He_j}{He_j}$
is finite. The set of all Hilbert-Schmidt operators will be denoted
$\hilbs H$. This space is a Hilbert space when given the inner
product
$\inp{H}{K} = \trace(HK^*)$.
Given a trace class operator $T$ there exist Hilbert-Schmidt operators
$H,K$ such that $T = HK^*$ and $\norm{T}_1^{1/2} = \norm{H}_2=
\norm{K}_2$. Let $A \in B(\hs H)$, let $T = HK^*\in TC(\hs H)$,
let $h_j = He_j$ and let $k_j = Ke_j$. If $h = \oplus h_j$
and $k = \oplus k_j \in \hs H \otimes \ell^2$, then,
\begin{align*}
  \trace(AT) &= \trace(AHK^*) = \trace(K^*AH) \\
  & = \sum_{j=1}^\infty \inp{AHe_j}{Ke_j} = \sum_{j=1}^\infty
  \inp{Ah_j}{k_j} = \inp{(A\otimes I)h}{k}.
\end{align*}
We also have
\[
\norm{h}^2 = \sum_{j=1}^\infty \norm{h_j}^2 = \sum_{j=1}^\infty
\inp{He_j}{He_j} = \trace(H^*H) = \norm{H}_2^2.
\]
Hence, $\norm{h} = \norm{H}_2$ and $\norm{k} = \norm{K}_2$.

Given an operator $A\in B(\hs H)$ and a $\wk$-closed subspace $\cal S
\subseteq B(H)$, the distance from $A$ to $\cal S$ is given by
\begin{align*}
  d(A,\cal S) = \inf\{\norm{A+S}\,:\, S\in\cal S\} =
  \sup\{\abs{\trace(AT)}\,:\, T\in \cal S_\perp, \norm{T}_1 = 1\}
\end{align*}
where $\cal S_\perp$ denotes the preannihilator of $\cal S$. If we
write $T = HK^*$ where $H, K \in HS(\hs H)$ and $\norm{T}_1^{1/2} =
\norm{H}_2=\norm{K}_2$, $h_j = He_j$, and $k_j = Ke_j$ as before,
then $\norm{h}=\norm{k} = 1$. Since $T = HK^* \in
\cal S_\perp$ we have
\[
0 = \trace(SHK^*) = \sum_{j=1}^\infty \inp{Sh_j}{k_j} = \inp{(S\otimes
  I)h}{k},
\] 
for all $S\in \cal S$.  Hence, $k\perp (\cal S\otimes I)h$. Rewriting
the distance formula we get,
\begin{equation}\label{dist1}
  d(A,\cal S) = \sup\{\abs{\inp{(A\otimes I)h}{k}}\,:\,
  \norm{h}=\norm{k} = 1,k \perp (\cal S \otimes I)h\}.
\end{equation}
It will also prove useful to have such a formula when $\cal S \subseteq
B(\hs H_1,\hs H_2)$. Let $\hs H = \hs H_1\oplus \hs H_2$. We can
identify $B(\hs H_1,\hs H_2)$ with a subspace of $B(\hs H)$ in the usual
way, that is, $A(h_1\oplus h_2) = 0\oplus Ah_1$. Keeping the notation
from \eqref{dist1} we can write $h\in (\hs H_1\oplus \hs H_2)\otimes
\ell^2$ as $h = h_1\oplus h_2$, where $h_i\in \hs H_i\otimes \ell^2$
for $i=1,2$. We have,
\[
\inp{(A\otimes I)h}{k} = \inp{0\oplus (A\otimes I)h_1}{k_1\oplus k_2}
= \inp{(A\otimes I)h_1}{k_2}.
\] 
Since $k\perp (\cal S\otimes I)h$ we see that $k_2 \perp (\cal S\otimes I)
h_1$. Hence,
\[
d(A,\cal S) = \sup\{\abs{\inp{(A\otimes
    I)h_1}{k_2}}\,:\,\norm{h_1}=1,\norm{k_2}=1, k_2\perp (\cal S
\otimes I)h_1\}.
\]

These computations prove the following distance formula:

\begin{thm}
\label{tdist2}
  Let $\cal S$ be a \wk-closed subspace of $B(\hs H_1,\hs H_2)$ and
  let $A\in B(\hs H_1,\hs H_2)$. The distance of $A$ from $\cal S$ is
  given by
\begin{equation}\label{dist2}
  d(A,\cal S) = \sup\{\abs{\inp{(A\otimes
      I)h_1}{h_2}}\,:\,h_i\in \hs H_i\otimes\ell^2, \norm{h_i}=1, h_2\perp (\cal S
  \otimes I)h_1\}.
\end{equation}
\end{thm}

\section{A Tangential interpolation theorem for subalgebras of $H^\infty$}
\label{maintangent}
\subsection{T\"oplitz corona problem}
We had made the claim at the end of Section~\ref{tangent} that the
multiplier algebra was \wk-closed and that point evaluations are
\wk-continuous. We now prove that claim.
\begin{lemma}\label{weakclosed}
  Let $K$ be a kernel on a set $X$ such that $K(x,x)\not= 0$ for all
  $x\in X$. Then the algebra $\mult(\hs H)$ is \wk-closed when viewed as a
  subalgebra of $B(\hs H(K))$. In addition, the evaluation map on $\mult(K)$
  given by $f\mapsto f(x)$ is \wk-continous.
\end{lemma}
\begin{proof}
  Let $M_{f_t}$ be a net that converges to an operator $T \in B(\hs
  H)$ in the \wk-topology. We have, $\inp{M_{f_t}h}{k}\to \inp{Th}{k}$
  for any pair of vectors $h,k\in \hs H$. Let $k = k_x$ and $h =
  k_y$. We have,
  \[\inp{Tk_y}{k_x} = \lim_t \inp{M_{f_t}k_y}{k_x} = \lim_t
  \inp{k_y}{M_{f_t}^*k_x} = \lim_t f_t(x)\inp{k_y}{k_x}. \] If we
  choose $x = y$, and use the fact that $K(x,x)\not = 0$, then $\lim_t f_t(x) =
  \frac{\inp{Tk_x}{k_x}}{K(x,x)}$. Hence, $\lim_t f_t(x)$ exists. Let
  us denote the limit by $f(x)$. We have $\inp{h}{T^*k_x} =
  f(x)\inp{h}{k_x}$ for all $h\in \hs H$ and $x\in X$, and so $T^*k_x
  = \cl{f(x)}k_x$. It follows that $f$ is a multiplier of $\hs H$ and
  that $T = M_f$.

  The above argument also shows that if $M_{f_t}\to M_f$ in the
  \wk-topology, then $f_t(x) \to f(x)$.
\end{proof}
We now show how the hypothesis of the T\"oplitz corona problem leads
to a tangential interpolation problem.  

Suppose that $F\in \mult(\hs H\otimes \hs C, \hs H)$ and that
$M_FM_F^\ast \geq \delta^2I$. We have, $M_F^*k_x = k_x\otimes F(x)^*$,
where we have identified $B(\hs{C},\bbC)$ with $\hs C$.  Hence,
\begin{align*}
  \inp{M_F^*k_y}{M_F^*k_x} & =  \inp{k_y\otimes F(y)^*}{k_x\otimes F(x)^*}\\
  & = \inp{F(y)^*}{F(x)^*}K(x,y)
\end{align*}  
Let $Y \subseteq X$ and let $\hs K_Y$ be the span of
$\{k_x\,:\, x\in Y\}$.  Since the space $\hs H$ is the closed linear
span of the set of kernel functions $k_x$ for $x\in X$, the operator
$M_FM_F^*-\delta^2I$ is positive if and only if $(M_FM_F^\ast -
\delta^2I)|_{\hs K_Y}$ is positive for all finite sets $Y\subseteq
X$. The latter condition is equivalent to the positivity of the matrix
\begin{equation}\label{matpos2}
  [(\inp{F(y)^*}{F(x)^*}-\delta^2)K(x,y)]_{x,y\in Y}
\end{equation}
for all finite sets $Y\subseteq X$.
\begin{prop}\label{tangtop}
  Let $\hs H$ be an RKHS on $X$ and let $\{K_\lambda :\lambda\in
  \Lambda\}$ be a set of kernels with the tangential interpolation
  property for $\mult(\hs H)$. Let $\hs H_\lambda = \hs H(K_\lambda)$,
  let $F\in \mult(\hs H \otimes \hs C, \hs H)$, and let
  $M_{F,\lambda}$ denote the operator of multiplication by $F$ between
  the spaces $\hs H_\lambda \otimes \hs C$ and $\hs
  H_\lambda$. Suppose that for each
  $\lambda\in\Lambda$, we have $M_{F,\lambda}M_{F,\lambda}^*\geq
  \delta^2I_\lambda$. Then there exists a multiplier $G\in \mult(\hs
  H,\hs H\otimes\hs C)$ such that $\norm{G}\leq \delta^{-1}$ and $FG =
  1$.
\end{prop}
\begin{proof}
  Since $M_{F,\lambda}M_{F,\lambda}^* \geq \delta^2 I_\lambda$, given
  a finite set $Y=\finset{x}$, \eqref{matpos2} shows that the matrix
  $[(\inp{F(x_j)^*}{F(x_i)^*}-\delta^2)K_\lambda(x_i,x_j)]\geq 0$ for
  all $\lambda\in \Lambda$. This matrix is of the form $Q_\lambda$ in
  \eqref{matpos} for the case where the vectors are $v_j = F(x_j)^*$,
  and the scalars $w_j = \delta$ for $j=1,\dots,n$. Since, $K_\lambda$
  has the tangential interpolation property, there is a contractive
  multiplier $G_Y\in \mult(\hs H,\hs H\otimes \hs C)$ such that
  $\delta = G_Y(x_j)^*F(x_j)^* = (F(x_j)G_Y(x_j))^*$ for
  $j=1,\ldots,n$. Since $\delta>0$ we get, $F(x_j)G_Y(x_j) = \delta$
  for $j=1,\ldots,n$. 

  The net $G_Y$ is contained in the unit ball of the space $\mult(\hs
  H,\hs H \otimes \hs C)$, which is $\wk$-compact subset of $B(\hs H,
  \hs H \otimes \hs C)$. If $\cal F$ denotes the collection of all
  finite subsets of $X$, then then there is a subnet $\{F_t\}\subseteq
  \cal F$ such that of $\{G_{F_t}\}$ converges in the $\wk$-topology
  to a contractive multiplier $G$. Since point evaluations are
  $\wk$-continuous on $\mult(\hs H, \hs H\otimes \hs C)$ we see, for a
  fixed $x\in X$, that $F(x)G(x) = \lim_{F_t} F(x)G_{F_t}(x) =
  \delta$. Therefore $FG = \delta$. It follows that $F(\delta^{-1}G) =
  1$ and $\norm{\delta^{-1}G}_{\mult(\hs H,\hs H\otimes \hs C)} \leq
  \delta^{-1}$.
\end{proof}

We have established that if $\{K_\lambda\,:\, \lambda\in \Lambda\}$ is
a set of kernels that have the tangential interpolation property, then
the condition $M_{F,\lambda}M_{F,\lambda}^*\geq \delta^2I_\lambda$
implies the existence of a multiplier $G\in \mult(\hs H,\hs H\otimes
\hs C)$ such that $\norm{G}\leq \delta^{-1}$ and $FG = 1$.

Our strategy for the proof of Theorem~\ref{interp} is to exploit the
distance formula and the existence of at least \textit{one} solution
to the tangential interpolation problem.

\subsection{The existence of solutions} 
Now let $\alg A$ be a unital $\wk$-closed subalgebra of the multiplier
algebra of $\hs H(K)$. Let $g\in \hs H$ be a nonvanishing function. In
particular, recall that an outer function $g$ does not vanish at any
point in the disk. We view $\mult(\hs H)$ as a subalgebra of $B(\hs
H)$. Let $\hs H_g$ be the closure of $\alg Ag$ in $\hs H$, let $K_g$
be the kernel of $\hs H_g$, let $k_x^g$ be the kernel function at the
point $x$, and let $Q_g =
[(\inp{v_j}{v_i}-w_i\cl{w_j})K_g(x_i,x_j)]$. We have assumed that
$\alpha = 1$. However, since this amounts to a rescaling, there is no
loss of generality in doing so. Since $g$ does not vanish at any point
$x\in X$, the kernel $K_g(x,x)\not =0$ for any $x$. Therefore, the
results of the previous section do apply.

We will establish the fact the positivity of the matrix $Q_g$ implies
the existence of a multiplier $F\in C(\alg A)$ such that $F(x_j)^*v_j
= \cl{w_j}$. We will then establish the fact that the closure of $\cal
J g$ in $\hs H_g\otimes \ell^2$ is the set of functions $f\in \hs
H_g\otimes \ell^2$ such that $\inp{f(x_j)}{v_j} = 0$ for
$j=1,\ldots,n$.

We say that the algebra $\alg A$ separates $x$ and $y$ if and only if
there exists a function $f\in \alg A$ such that $f(x)\not = f(y)$. We
say that the algebra $\alg A$ separates a set of points $Y$ if and
only if $\alg A$ separates $x$ and $y$ for all $x,y\in Y$.

\begin{lemma}
Every element of $\cal A$ is a multiplier of $\hs H_g$.
\end{lemma}
\begin{proof}
  Let $h\in \hs H_g$ and let $f_n$ be a sequence in $\calA$ such that
  $\norm{f_ng-h}_{\hs H} \to 0$. Let $f\in \calA$. We have
  $\norm{(ff_n)g - fh}_{\hs H} \leq \norm{M_f}\norm{f_ng-h}_{\hs H}
  \to 0$. Hence, $fh \in \hs H_g$.
\end{proof}

\begin{lemma}\label{linindep}
  The algebra $\alg A$ separates $x$ and $y$ if and only if $k_x^g$ and
  $k_y^g$ are linearly independent.
\end{lemma}
\begin{proof}
Suppose that $\alg A$ does separate $x,y$
and that $f\in \alg A$ with $f(x)=1$ and $f(y)=0$. Note that $f$ is a
multiplier of $\hs H_g$. Assume that $\alpha k_x^g + \beta k_y^g = 0$. We have
\[ 0 = M_f^\ast (\alpha k_x^g +\beta k_y^g) = \alpha \cl{f(x)}k_x^g + \beta
\cl{f(y)}k_y^g = \alpha k_x^g.\] Since $g$ is a nonvanishing function in
$\hs H_g$ we know that $k_x^g\not =0$ and so $\alpha = 0$. A similar
argument shows that $\beta = 0$.

Conversely, suppose that $\alg A$ does not separate $x$ and $y$. Let
$z\in X$ and let $f_n$ be a sequence in $\alg A$ such that $f_ng\to
k_z^g$. We have $k_z^g(x) = \lim_{n\to\infty} f_n(x)g(x)$. On the
other hand, $f_n(x)=f_n(y)$ and so 
\begin{align*}
k_z^g(y) &= \lim_{n\to\infty}f_n(y)g(y) = \lim_{n\to\infty} f_n(x)g(y) \\
& = \lim_{n\to\infty} f_n(x)g(x)\frac{g(y)}{g(x)} = k_z^g(x)\frac{g(y)}{g(x)}.
\end{align*}  
Hence, $g(x)K_g(y,z) -g(y)K_g(x,z) = 0$ for all $z\in X$ and so
$\cl{g(x)}k_y^g - \cl{g(y)}k_x^g = 0$ with $g(x),g(y)\not = 0$.
\end{proof}
The relation $x\sim y$ if and only if $f(x)=f(y)$ for all
$f\in \alg A$ is an equivalence relation on $X$. Let
$\{x_1,\ldots,x_n\}$ be given. Let us reorder the points
$\{x_1,\ldots,x_n\}$ in such a way that there is a sequence
$n_0=0<n_1<\cdots<n_p=n$ such that the sets $X_k = \{x_i\,:\, n_{k-1}<
i \leq n_k\}$ are the equivalence classes for the above equivalence
relation.

\begin{lemma}\label{pointsep}
  If $Q_g\geq 0$, then there exists a multiplier $F\in C(\cal A)$
  such that $F(x_j)^*v_j = \cl{w_j}$ for $j=1,\ldots,n$. In addition, the
  subspace $[\cal Jg]$ is precisely the set of functions in $\cal
  H_g\otimes \ell^2$ such that $\inp{f(x_j)}{v_j} = 0$ for
  $j=1,\ldots,n$
\end{lemma}
\begin{proof}
  By Lemma~\ref{linindep} there exist functions $e_1,\ldots,e_p$ such
  that $e_k|_{X_l}(x) = \delta_{k,l}$ for $k,l=1,\ldots,p$.

  To simplify notation let $K = K_g$, let $Q =
  [(\inp{v_j}{v_i}-w_i\cl{w_j})K(x_i,x_j)]_{i,j=1}^n = [q_{i,j}]$ and
  let $Q_k = [q_{i,j}]_{n_{k-1}<i,j\leq n_k}$.  Let $k$ be given and
  let $t = n_k-n_{k-1}$. We temporarily set $Y = X_k$ and relabel the
  points of the set $X_k$ as $y_1,\ldots,y_t$. We also relabel the
  corresponding vectors as $v_1,\ldots,v_t$ and the scalars
  $w_1,\ldots,w_t$. Let $e = e_k$. 

  Since the algebra $\alg A$ fails to separate any two points of $Y$
  we see that there exists a sequence of nonzero scalars
  $\lambda_1,\ldots,\lambda_{t}$ such that $k_{y_i} = \lambda_i
  k_{y_1}$. The matrix $Q_k$ is given by
  \[Q_k
  =[(\inp{v_j}{v_i}-w_i\cl{w_j})K(y_1,y_1)\lambda_i\cl{\lambda_j}].\]
  Since this is a square submatrix of $Q$ we know that $Q_k\geq
  0$. Since the $\lambda_i$ are nonzero and $K(y_1,y_1)$ is nonzero we
  see that $[\inp{v_j}{v_i}]\geq [w_i\cl{w_j}]$. Hence, the vector
  $(w_1,\ldots,w_t)^t$ is in the range of the matrix $P =
  [\inp{v_j}{v_i}]$. Therefore, there are scalars
  $\alpha_1,\ldots,\alpha_t$ such that $\inp{\sum_{j=1}^t \alpha_j
    v_j}{v_i} = w_i$ for $i=1,\ldots,t$. Let $\xi = \sum_{j=1}^t
  \alpha_j v_j \in \ell^2$. We let $\xi_i$ denote the $i$th component
  of $\xi$. Consider the function $F = (\xi_1e,\xi_2e,\ldots)^t$,
  which belongs to $C(\alg A)$, because $\xi\in \ell^2$. We have that
  $F(x) = \xi$ if $x\in Y$ and is $0$ if $x\in
  \{x_1,\ldots,x_n\}\setminus Y$.

  From the argument in the previous paragraph we see that for each $
  k$ such that $1\leq k \leq p$ we can find $\xi_k$ such that
  $\inp{\xi_k}{v_i} = w_i$ for $n_{k-1}<i\leq n_k$. In addition, we
  can find $F_k$ such that $F_k(x) = \xi_k$ for $x\in X_k$ and $F_k(x)
  = 0$ if $x\in \{x_1,\ldots,x_n\}\setminus X_k$. Hence,
\[
F_k(x)^*v_i =
\begin{cases}
  \cl{w_i} &\text{ if }x\in X_k\\
  0& \text{ if }x\in \{x_1,\ldots,x_n\}\setminus X_k
\end{cases}.
\] 
Hence, the function $F = F_1+\cdots+F_p$ has the property that
$F(x_i)^*v_i = w_i$ for $i=1,\ldots,n$.

Let $\cal J$ be the set of functions $F\in C(\alg A)$ such that $F(x_j)^*v_j
= 0$ for $j=1,\ldots,n$. We claim that $[\cal Jg]$ is the set of functions
in $\hs H_g\otimes \ell^2$ such that $\inp{f(x_j)}{v_j} = 0$. 

One inclusion is straightforward. If $F_mg \to h$, then
\[\inp{v_j}{h(x_j)} = \lim_{m\to \infty}\inp{v_j}{F_m(x_j)g(x_j)} =
\cl{g(x_j)}\lim_{m\to\infty}F_m(x_j)^*v_j = 0.\]

The reverse inclusion is a little more involved. Let $f\in
[\cal Jg]$. There exists a sequence $F_m \in C(\alg A)$ such that
$\norm{F_mg-f}\to 0$. We need to modify $F_m$ to a sequence
$\tilde{F_m}$ such that $\norm{\tilde{F_m}g-f}\to 0$ and
$\tilde{F_m}\in \cal J$. Once again let $X_1,\ldots,X_p$ be the partition
of the set $\{x_1,\ldots,x_n\}$ given by the equivalence relation of
point separation.

Given a function $a\in \alg A$ we define $a\otimes \xi $ to be the
function in $C(\alg A)$ given by $(a\otimes \xi) h = ah\otimes
\xi$. Let $\eta_{k,m}$ be the orthogonal projection of the vectors
$F_m(x_{n_k})$ onto the finite-dimensional subspace spanned by
$\{v_i\,:\, n_{k-1}<i\leq n_k\}$. Since, $\norm{F_mg-f}\to 0$, we have
that $\inp{F_m(x_i)g(x_i)-f(x_i)}{v_i} = \inp{F_m(x_i)}{v_i}\to 0$ for
$i=1,\ldots,n$. Hence, $\norm{\eta_{k,m}}\to 0$ as $m \to \infty$.

Let $G_m = \sum_{k=1}^p e_k\otimes \eta_{k,m}$, let $1\leq i\leq n$,
and let $l$ be such that $n_{l-1}<i\leq n_l$. We have,
\[
G_m(x_i)^*v_i = \sum_{k=1}^p (e_k\otimes \eta_{k,m})(x_i)^*v_i =
\sum_{k=1}^p \cl{e_k(x_i)}\inp{v_i}{\eta_{k,m}}.
\] 
Since $e_k|_{X_l}(x) = \delta_{k,l}$, the terms in the above sum for
$k\not=l$ are zero. Hence, the sum reduces to
$\inp{v_i}{\eta_{l,m}}$. Recall that $\eta_{l,m}$ is the projection of
$F(x_{n_l})$ onto the span of the vectors $\{v_i\,:\, n_{l-1}<i\leq
n_l\}$. Hence, $\inp{v_i}{\eta_{l,m}} = \inp{v_i}{F_m(x_{n_l})} =
F_m(x_{n_l})^*v_i$. However, functions in $\alg A$ are constant on the
sets $X_k$, which means $F_m(x_{n_l}) = F_m(x_i)$ and we get
$G_m(x_i)^*v_i = F_m(x_i)^*v_i$ for all $i=1,\ldots,n$. Hence, the
function $\tilde{F_m} = F_m - G_m\in \cal J$.

We have,
\[\norm{G_mg} = \norm{\sum_{k=1}^p
  (e_k\otimes F_m(x_{n_k}))g} \leq \sum_{k=1}^p\norm{e_kg\otimes
  \eta_{k,m}} =\sum_{k=1}^p \norm{e_kg}\norm{\eta_{k,m}}\to 0
\]
as
$m\to \infty$.

Finally, note that $f\in \hs H_g \otimes \ell^2$ is orthogonal to
$k_x\otimes v$ if and only if $\inp{f(x)}{v} = 0$. Hence, $(\hs H_g
\otimes \ell^2) \ominus [\cal J g]$ is the span of the vectors
$\{k_{x_i}\otimes v_i\,:\,i=1,\ldots,n\}= \hs K_g$.
\end{proof}

The distance formula, equation \eqref{dist2}, in Theorem~\ref{tdist2}
shows that we must be able to classify the cyclic subspaces of the
form $(\calA \otimes I)h$, and to do so, we begin with a simple
lemma. We will use the natural identification between $\hs H\otimes
\ell^2$ and the $\ell^2$-direct sum of $\hs H$, which we denote
$\oplus_{j=1}^\infty \hs H$.
\begin{lemma}
  Let $\alg A \subseteq B(\hs H)$ and let $h\in \hs H$, then the cyclic
  subspace of $\hs H$ generated by $C(\alg A)$ and $h$ is equal to
  $[\alg Ah]\otimes \ell^2$.
\end{lemma}
\begin{proof}
  The space $C(\alg A)h$ is generated by elements of the form
  $ah\otimes e_j$ for $j\in \bbN$. Hence, $[\alg Ah]\otimes
  \ell^2\subseteq [C(\alg A)h]$. On the other hand, an element of
  $C(\alg A)h$ is of the form $\oplus_{j=1}^\infty a_jh$ where
  $\sum_{j=1}^\infty \norm{a_jh}^2$ is finite and hence is in $[\alg
  Ah]\otimes \ell^2$.
\end{proof}
\subsection{Proof of Theorem \ref{interp}}
We will now prove our main theorem, Theorem \ref{interp}, which is a
tangential interpolation result for $\wk$-closed subalgebras of
$H^\infty$. Let $\alg A\subseteq H^\infty$ be a unital \wk-closed
subalgebra of $H^\infty$.  

So far, we have assumed no additional structure on the algebras.  A
function $g \in H^2$ is called outer if and only if $H^\infty g$ is
dense in $H^2$. When $\alg A$ is subalgebra of $H^\infty$ the space
$(C(\alg A)\otimes I)h$, which is contained in $(H^2\otimes
\ell^2)\otimes \ell^2$, can be identified with a subspace of
$H^2\otimes \ell^2$ of the form $C(\alg A)g$ for some outer
function $g$. 
\begin{lemma}\label{outer}
  Let $\{h_i\}_{i=1}^{\infty}$ be a sequence in $H^2$ such that
  $\sum_{i=1}^\infty \norm{h_i}^2$ is finite. Then the function $p(t)
  = \sum_{i=1}^\infty \abs{h_i(t)}^2\in L^1(\bbT)$ and there exists an
  outer function $g\in H^2$ such that $p = \abs{g}^2$ a.e.~$\bbT$.
\end{lemma}
\begin{proof}
The fact that $p\in L^1(\bbT)$ is a straightforward argument. 

It is well-known that a non-negative function $p\in L^1(\bbT)$ is of
the form $p=\abs{g}^2$ for some outer function if and only if the
function $\log p$ is summable. Let $p_m = \sum_{i=1}^m
\abs{h_i}^2$. If $u_1$ denotes the outer part of $h_1$, then $p_1 =
\abs{h_1}^2 = \abs{u_1}^2$. Hence $\log p_1\in L^1$. Since $\log p <
p$ it follows, since $p\in L^1$, that $(\log p)^+\in L^1$. On the
other hand we have $\log p \geq \log p_1$ and so $(\log p)^-\leq (\log
p_1)^-$. It follows that $(\log p)^-\in L^1$.
\end{proof}

\begin{lemma}\label{reduction}
  Let $\oplus_{i=1}^\infty h_i \in H^2\otimes \ell^2$ and let $g$ be
  as in Lemma~\ref{outer}. If $\alg A$ is a unital subalgebra of
  $H^\infty$, and $h = \sum_{j=1}^\infty h_j\otimes e_j$, then the map
  $U:[(C(\alg A)\otimes I)h]\to [C(\alg A)g]$ defined by
  $U[((M_F\otimes I)h)] = [M_Fg]$ is a unitary operator.
\end{lemma}

\begin{proof}
  The map is clearly linear and surjective. We now show that $U$ is
  isometric, which also proves that $U$ is well-defined.

  We have,
\begin{align*}
  \norm{M_Fg}^2 &= \int \norm{Fg}^2 = \int \sum_{i=1}^\infty \abs{f_ig}^2  \\
  & = \int \sum_{i=1}^\infty \abs{f_i}^2 (\sum_{j=1}^\infty
  \abs{h_j}^2) =
  \sum_{j=1}^\infty \sum_{i=1}^\infty \int \abs{f_ih_j}^2 \\
  & = \sum_{j=1}^\infty \norm{M_Fh_j}^2 = \norm{(M_F\otimes I)h}^2.
\end{align*}
\end{proof}

We are now in a position to prove our main result, which is a
tangential interpolation theorem for subalgebras of $H^\infty$. Given
an outer function we will denote by $\hs H_g$ the closed subspace
$H^2$ generated by elements of the form $fg$, where $f\in \alg A$. We
will denote by $K_g$ the kernel function for this subspace. When $g$
is the constant function $g\equiv 1$ we will suppress the subscript
$g$.

\begin{proof}[Proof of Theorem \ref{interp}]
  We have seen that the existence of a solution $F$ of norm at
  most $\alpha$ implies $Q_g\geq 0$ for all $g$. Hence, it is the
  converse that concerns us.

  Let us assume that there is at least one solution, say $F_0$, which exists by Lemma~\ref{pointsep}. We view $\alg A$
  as a subalgebra of $B(\hs H)$. We define $\cal J$ to be the set of
  functions in $C(\calA)$ such that $F(x_j)^*v_j =0$ for
  $j=1,\ldots,n$. 

  Applying the distance formula~\eqref{dist2}, we get
  \[
  d(F_0,\cal J) = \sup\{\abs{\inp{(M_{F_0}\otimes I) h}{k}}\},
  \] 
  where $h\in \hs H\otimes \ell^2$, $k\in (\hs H\otimes \ell^2)\otimes
  \ell^2$, $\norm{h}=\norm{k}=1$ and $k\perp (\cal J\otimes I)h$. By
  projecting onto the subspace $[(C(\alg A)\otimes I)h]$ we can assume
  that $k\in [(C(\alg A)\otimes I)h]\ominus [(\cal J\otimes I)h]$. Let
  $U$ be the unitary map from Lemma~\ref{reduction} and let $g$ be the
  outer function such that $\abs{g}^2 = \sum_{i=1}^\infty
  \abs{h_i}^2$. We have,
  \begin{align*}
    \inp{M_{F_0}h}{k} = \inp{UM_{F_0}h}{Uk} = \inp{M_{F_0}g}{Uk} =
    \inp{M_{F_0}g}{k'},
  \end{align*}
  where $k' = Uk$. This gives
  \[
  d(F_0,\cal J) = \sup\{\abs{\inp{M_{F_0}g}{Uk}}\} =
  \sup\{\abs{\inp{M_{F_0}g}{k'}}\},
  \]
  where the supremum is over all outer functions $g\in H^2$ such that
  $\norm{g}=1$, $\abs{g}^2 = \sum_{j=1}^\infty \abs{h_i}^2$,
  $\norm{k'}\leq 1$, and $k'\in [C(\alg A)g]\ominus [\cal J g]$.

  Lemma~\ref{pointsep} shows that $[\cal J g]$ is the set of
  functions in $f\in \hs H_g\otimes \ell^2$ such that
  $\inp{f(x_j)}{v_j} = 0$ for $j=1,\dots,n$.  Since
  $\inp{f}{k_x\otimes \xi} = \inp{f(x)}{\xi}$, we see that $\hs K_g :=
  [C(\alg A)g]\ominus [\cal J g]$ is the span of the vectors
  $\{k^g_{x_i}\otimes v_i\,:\, i=1,\ldots,n\}$, where $k^g_x$ is the
  kernel function for $\hs H_g$ at the point $x$.

  Therefore,
  \[d(F_0,\cal J) \leq \sup \abs{\inp{g}{M_{F_0}^*k'}}\leq \sup
  \norm{M_{F_0}^*|_{\hs K_g}}\] where the supremum is taken over all
  outer functions $g$ as above.  If $F\in \cal J$, then
  $M_F^*(k^g_{x_i}\otimes v_i) = F(x_i)^*v_i k_{x_i}^g = 0$, and so
  $M_F^\ast|_{\cal K_g} = 0$. Hence, $\norm{F_0+F} \geq
  \norm{M_{F_0+F}^*|_{\hs K_g}} = \norm{M_{F_0}^*|_{\hs K_g}}$ from
  which it follows that $d(F_0,\cal J) = \sup_g \norm{M_{F_0}^*|_{\hs
      K_g}}$.

  The calculation leading to~\eqref{matpos3} shows that
  $\norm{M_{F_0}^*|_{\hs K_g}}\leq \alpha$ if and only if the matrix
  $Q_g\geq 0$. Hence, the positivity of all the matrices $Q_g$ implies
  that $d(F_0,\cal J)\leq \alpha$. This in turn guarantees the
  existence of a solution to the tangential problem of norm at most
  $\alpha$.
\end{proof}
The proof of Corollary~\ref{interpcor} is a consequence of the following
observation. 

If $h\in \hs H$, then there exists a sequence $f_n\in \calA$ such that
$\norm{f_n-h}_2 \to 0$. Hence, $\norm{\abs{f_n}^2 - \abs{h}^2}_1 \leq
\norm{f_n-h}_2\norm{f_n+h}_2 \to 0$ as $n\to \infty$. Therefore
$\abs{h}^2 \in L^1(\calA)$. If $(h_n)\in \hs H\otimes \ell^2$ is a
square summable sequence, then $\sum_{n=1}^\infty \abs{h_n}^2 \in
L^1(\calA)$. Hence, the absolute value of the outer function $g$ such
that $\abs{g}^2 = \sum_{n=1}^\infty \abs{h_n}^2$ is an element of
$L^2(\calA)$.

We have now established a tangential interpolation theorem for
\wk-closed subalgebras of $H^\infty$. Proposition~\ref{tangtop} shows
that the tangential interpolation result implies a T\"oplitz corona
theorem.  This result can be viewed as analogous to the results
obtained in Trent-Wick~\cite{TW} and
Douglas-Sarkar~\cite{DS}. However, the method of proof is different.

\subsection{Examples}
A better feeling for the result in Theorem~\ref{interp} can be
obtained by examining some special cases. We single out two classes of
examples of subalgebras of $H^\infty$.

\begin{enumerate}
\item Let $B$ be an inner function and consider the algebra
  $H^\infty_B = \bbC+B H^\infty$. Note that $B\in H^\infty_B$ and so
  $\cl{B}\in L^\infty(H^\infty_B)$. Therefore, $H^\infty =
  \cl{B}BH^\infty$ is contained in $L^\infty(H^\infty_B)$ and we see
  that $L^\infty(H^\infty_B) = L^\infty$.

We can also give a more explicit description of the subspaces $\hs
H_g$ in this case. Let $g$ be an outer function and let $v =
P_{H^2\ominus BH^2}g$. We claim that $\hs H_g = [v]\oplus BH^2$. We have, 
\[[(\bbC+BH^\infty)g] = \bbC g + B[H^\infty g] = \bbC v \oplus BH^2.\] 

\item Let $R$ be finite open Riemann surface. It is well-known that the
universal covering space for $R$ is the open unit disk $\bbD$. Let
$p:\bbD \to R$ denote the covering map and let $\Gamma$ denote the set
of deck transformations, that is, automorphisms $\gamma$ of the disk
such that $p\circ \gamma = p$.

The automorphisms in the group $\Gamma$ act on the disk and induce an
action on the space $H^\infty$ by composition. The set of fixed points
$H^\infty_\Gamma$ is naturally identified with the space of bounded
holomorphic functions on the Riemann surface. The automorphisms also
act by bounded linear maps on the space $H^p$ and $L^p$ and we use a
subscript $\Gamma$ to denote the associated set of fixed points.

In this case the algebra $L^\infty(H^\infty_\Gamma) =
L^\infty_\Gamma$.

The cyclic subspace $\hs H_g$ for an outer function $\abs{g}\in
L^2_\Gamma$ can be described in terms of character automorphic
function. A character of $\Gamma$ is a homomorphism from $\Gamma$ into
the circle group $\bbT$ and we denote the space of characters by
$\hat{\Gamma}$. A function $h\in H^2$ is called character automorphic
if there exists a character $\sigma\in \hat{\Gamma}$ such that $h\circ
\gamma = \sigma(\gamma) h$. The closure of $H^\infty_\Gamma$ in $H^2$
is the space $H^2_\Gamma$. If $g$ is an outer function such that
$\abs{g} \in L^2(\cal A)$, then there exists a character $\sigma\in
\hat{\Gamma}$ such that $g\circ \gamma = \sigma(\gamma)g$. In
addition, the space $\hs H_g = H^2_\sigma := \{f\in H^2\,:\, f\circ \gamma
= \sigma(\gamma)f\}$. A proof of these facts can be found in~\cite{Ra2}.

Given a character $\sigma$ we let $K^\sigma$ denote the reproducing kernel of
$H^2_\sigma$. We get that the kernels of
the character automorphic spaces $H^2_\sigma$, where $\sigma\in
\hat{\Gamma}$ have the tangential interpolation property for $H^\infty_\Gamma$.

In Section~\ref{similar} we will return to this example and show that
we can replace this family of matrix positivity conditions by a single
condition, at the expense of the optimal constant for the norm of a
solution to the tangential interpolation problem.
\end{enumerate}

We point out that the tangential interpolation theorem gives us a new
way to derive the Nevanlinna-Pick type interpolation results
in~\cite{Ra1,Ra2} for the examples above. 

\section{Applications of Theorem \ref{onegenerator}:  Similar Cyclic Modules}
\label{similar}
Theorem~\ref{interp} shows that the positivity of $M_FM_F^*\geq
\delta^2$ on a family of reproducing kernel Hilbert spaces is enough
to guarantee the existence of a function $G$ such that $FG = 1$ and
$\norm{M_G} \leq \delta^{-1}$. This theorem is analogous to the
results obtained in the work of~\cite{A,TW,DS}.

If we drop the requirement that the function $G$ have optimal norm,
then in some cases we can replace the family of conditions by a single
condition. Let $\calI_x$ denote the ideal of functions in $\calA$ such
that $f(x)=0$. We have already seen that $[\calI_x g]$ is a
codimension one subspace of $\hs H_g$ and that the orthogonal complement
of $\calI_xg$ is spanned by the kernel function $k^g_x$.

Now let us return to the setting where $\cal A$ is a unital \wk-closed
subalgebra of $H^\infty$ and the function $g$ is outer. We will
establish a tangential interpolation theorem where we replace the
family of conditions $Q_g\geq 0$ for all outer functions $g$ such that
$\abs{g}\in L^2(\cal A)$ by a single positivity condition. However, we
can not guarantee a solution of optimal norm. We will then apply our
result to the case of finite open Riemann surfaces.

Let $g,h$ be two outer functions and let $S:\hs H_g \to \hs H_h$ be a
bounded invertible operator that intertwines the action of $\calA$, that is,
such that $SM_f = M_fS$ for all $f\in \calA$. By taking adjoints we
see that $M_f^* S^* = S^* M_f^*$ for all $f\in \calA$. If $x\in \bbD$,
then $M_f^* k^g_x = \cl{f(x)}k^g_x$ and so we have $M_f^*S^*k^h_x =
\cl{f(x)}S^*k_x^h$ for all $x\in \bbD$ and $f\in \calA$. It follows
that the vector $S^*k^h_x$ is orthogonal to $\calI_x g$. Hence,
$S^*k^h_x = \cl{\phi(x)}k^g_x$ for all $x\in \bbD$, where $\phi$ is a
complex-valued function on the disk. In fact, $\phi$ is a multiplier
from $\hs H_g \to H_h$ and $S = M_\phi$.

Let $\hs H_1$, $\hs H_2$, $\hs K_1$ and $\hs K_2$ be $n$-dimensional
Hilbert spaces. If $A \in B(\hs H_1 ,\hs H_2)$ and $x_1,\ldots,x_n$ is
a basis for the space $\hs H_1$, then $\norm{A}\leq \alpha$ if and only if the
matrix $[\alpha^2\inp{x_j}{x_i}-\inp{Ax_j}{Ax_i}]$ is positive
semidefinite. Now let $S_i\in B(\hs H_i,\hs K_i)$ be a bounded
invertible transformation. If $\norm{A}\leq \alpha$, then
$\norm{S_2AS_1^{-1}}\leq \alpha \norm{S_2}\norm{S_1^{-1}}$.

Consider the special case where $\hs H_1 = \linspan\{k_{x_1}^g\otimes
v_1,\ldots,k_{x_n}^g\otimes v_n\}$, $\hs K_1 =
\linspan\{k_{x_1}^h\otimes v_1,\ldots,k_{x_n}^h\otimes v_n\}$, $\hs
H_2 = \linspan\{k_{x_1}^g,\ldots,k_{x_n}^g\}$, $\hs K_2 =
\linspan\{k_{x_1}^h,\ldots,k_{x_n}^h\}$. Let $A$ be the map
$A(k_{x_i}^g\otimes v_i) = \cl{w_i} k_{x_i}^g$, let $S = M_\phi$ be the
similarity between $\hs H_g$ and $\hs H_h$ described above, let $S_1 =
S^*\otimes I$, and let $S_2 = S^*$.

Note that if $F\in C(\cal A)$, with $F(x_i)^* v_i = \cl{w_i}$, then
$M_F^*|_{\hs H_1}$ is precisely $A$.

If $\norm{S}\norm{S^{-1}} = c$, then a straightforward computation
shows that $[(\alpha^2\inp{v_j}{v_i}-w_i\cl{w_j})K_g(x_i,x_j)]\geq 0$
if and only if $\norm{A}\leq \alpha$. This implies $\norm{S^*A(
  (S^*)^{-1}\otimes I)} \leq c\alpha$ which in turn implies that
$[(c^2\alpha^2\inp{v_j}{v_i}-w_i\cl{w_j})K_h(x_i,x_j)]\geq 0$.

We are now in a position to prove Theorem~\ref{onegenerator}.

\begin{proof}[Proof of Theorem \ref{onegenerator}]
  From the observations made above we see that the matrix positivity
  condition implies that
  $[(\alpha^2c^2\inp{v_j}{v_i}-w_i\cl{w_j})K^g(x_i,x_j)]\geq 0$. The
  result in (1) now follows from Theorem~\ref{interp}.

  The proof of (2) follows, as before, from the tangential
  interpolation theorem established in (1).
\end{proof}

We now provide an example of a class of subalgebras of $H^\infty$ to
which the above theorem applies. Recall that if $R$ is a finite open
Riemann surface, and $\Gamma$ is the associated group of deck
transformations acting on the disk, then the fixed-point algebra
$H^\infty_\Gamma$ is naturally identified with $H^\infty(R)$.

In this case the outer function $g$ has the property that there is a
character $\sigma$ such that $g\in H^2_\sigma$, and $\hs H_g =
H^2_\sigma$.  We will establish the existence of a similarity
$S_\sigma = M_{\phi_\sigma}$ between $H^2_\Gamma$ and $H^2_\sigma$ and
show that there is a uniform bound on
$\norm{S_\sigma}\|S_\sigma^{-1}\|$. This result generalizes a theorem
of Ball~\cite{B} from the setting of multiply-connected domains to
Riemann surfaces.

\begin{prop}\label{similarity}
  Let $\Gamma$ be a the group of deck transformations associated to a
  finite open Riemann surface. For each $\sigma\in \hat{\Gamma}$,
  there exists a bounded invertible function $\phi_\sigma$ such that
  $\phi_\sigma H^2_\Gamma = H^2_\sigma$. There is a constant $\beta$,
  independent of $\sigma$ such that $\beta^{-1}\leq
  \abs{\phi_\sigma}\leq \beta$.
\end{prop}

We will need two results of Forelli~\cite{F}.

\begin{thm}[Forelli] Let $\Gamma$ be the group of deck transformations
  associated to a finite open Riemann surface $R$ of genus $m$. Let
  $\gamma_1,\ldots,\gamma_m$ denote the generators of $\Gamma$. There
  exist $m$ vectors $v_1,\ldots,v_m \in L^\infty_\Gamma$ such that
  $v_i$ is non-negative, and $v_i$ is orthogonal to $H^2_\Gamma\oplus
  \cl{H^2_{\Gamma,0}}$. In addition, $v_1,\ldots,v_m$ are linearly
  independent.
\end{thm}

If $f$ is a real-valued function in $L^2$, then its conjugate function
$f^*$ is the unique real-valued function in $L^2$ such that $f+if^*\in H^2$ and
$\int f^* = 0$. 

\begin{thm}[Forelli]
  Let $f\in L^2_\Gamma$ and let $f^*$ denote the function conjugate to
  $f$, then $f^*\circ \gamma_i - f$ is constant, and the constant is
  given by $\int fv_i$.
\end{thm}

Now we present the proof of Proposition~\ref{similarity}.

\begin{proof}[Proof of Proposition \ref{similarity}]
  Let $\gamma_1,\ldots,\gamma_m$ be a minimal set of generators of the
  group $\Gamma$. We let $\sigma_k = \sigma(\gamma_k)$. Since $\sigma$
  is a character, there exists $\theta_1,\ldots,\theta_m \in [0,2\pi)$
  such that $\sigma_k = e^{i\theta_k}$. 

  Let $v_{k,l} = \inp{v_l}{v_k}$ for $k,l=1,\ldots,m$.  Note that
  $v_l^\ast\circ \gamma_k = v_l^* + \int v_kv_l = v_l^* +
  \inp{v_k}{v_l}$. Since $v_1,\ldots,v_m$ are linearly independent the
  matrix $V = [v_{k,l}]$ is invertible. Let $c = (c_1,\ldots,c_m)^t$
  be the unique vector such that $ Vc =
  (\theta_1,\ldots,\theta_m)^t$. Since the entries of $V$ are real we
  see that $V^{-1}$ has real entries. Hence, the vector $c\in \bbR^m$.

  Let $f = \sum_{k=1}^m c_k v_k$. Note that $f$ is a real-valued
  element of $L^\infty$. Following the construction in~\cite{F} we let
  $\phi_\sigma = \exp(f+if^*)$. Since $f$ is real-valued we see that
  $\abs{\phi_\sigma} = \exp(f)$ and so $\phi_\sigma$ is bounded. Now,
  \begin{align*}
  \phi_\sigma\circ \gamma_k &= \exp\bigg(f+if^* + i\sum_{l=1}^m c_lv_{k,l} \bigg) \\
  &=  \exp\bigg(i\sum_{l=1}^m c_l v_{k,l}\bigg)\phi_\sigma = \exp(i\theta_k)\phi_\sigma =
  \sigma(\gamma_k)\phi_\sigma.
  \end{align*} Hence, $\phi_\sigma\in H^\infty_\sigma$.

  We have $\abs{\sum_{k=1}^m c_k v_k} \leq
  \max_{k=1,\ldots,m}\norm{v_k}_\infty \norm{c}_1$. Since
  $\theta_1,\ldots,\theta_m\in [0,2 \pi)$ there exists a constant $K$,
  that does not depend on $\sigma$, such that $\norm{c}_1\leq
  K$. Hence, there is a constant $K'$ such that $e^{-K'}\leq
  \abs{\phi_\sigma} \leq e^{K'}$ for all $\sigma\in \hat{\Gamma}$.
\end{proof}
%

\end{document}